\documentclass[12pt,sort&compress]{elsarticle}

\makeatletter
\def\ps@pprintTitle{\let\@oddhead\@empty
  \let\@evenhead\@empty
  \def\@oddfoot{\reset@font\hfil\thepage\hfil}
  \let\@evenfoot\@oddfoot
}
\makeatother

\usepackage{amsmath}
\usepackage{amsfonts}
\usepackage{amssymb}
\usepackage{amsfonts}
\usepackage{amsthm}
\usepackage{amscd}
\usepackage{mathtools}
\usepackage{commath}
\usepackage{lmodern}
\usepackage{color}
\usepackage{a4wide}
\usepackage{bm} 
\usepackage{todonotes}

\usepackage[draft]{fixme}
\fxsetup{theme=color}
\usepackage{hyperref}

\usepackage[czech,english]{babel}
\usepackage[IL2]{fontenc}
\usepackage[utf8]{inputenc}

\newtheorem{theorem}{Theorem}[section]
\newtheorem{lemma}[theorem]{Lemma}

\newtheorem{corollary}[theorem]{Corollary}
\newtheorem{problem}[theorem]{Problem}
\theoremstyle{definition}
\newtheorem{definition}[theorem]
{Definition}
\theoremstyle{remark}
\newtheorem{remark}[theorem]{\upshape\bfseries Remark}
\newtheorem{example}[theorem]{\upshape\bfseries Example}

\newcommand{\C}{\mathbb{C}}
\newcommand{\R}{\mathbb{R}}
\renewcommand{\H}{\mathbb{H}}
\newcommand{\Z}{\mathbb{Z}}

\newcommand*{\qi}{\mathbf{i}}
\newcommand*{\qj}{\mathbf{j}}
\newcommand*{\qk}{\mathbf{k}}

\newcommand*{\Cj}[1]{{#1}^\ast}

\begin{document}

\begin{frontmatter}

\title{Partial Fraction Decomposition\\ for Rational Pythagorean Hodograph Curves}

\author[1]{Hans-Peter Schröcker}
\ead{hans-peter.schroecker@uibk.ac.at}
\address[1]{Universität Innsbruck, Department of Basic Sciences in Engineering Sciences, Technikerstr.~13, 6020 Innsbruck, Austria}

\author[2]{Zbyněk Šír}
\ead{zbynek.sir@karlin.mff.cuni.cz}
\address[2]{Charles University, Faculty of Mathematics and Physic, Sokolovská 83, Prague 186 75, Czech Republic}

\begin{abstract}
    All rational parametric curves with prescribed polynomial tangent direction form a vector space. Via tangent directions with rational norm, this includes the important case of rational Pythagorean hodograph curves. We study vector subspaces defined by fixing the denominator polynomial and describe the construction of canonical bases for them. We also show (as an analogy to the fraction decomposition of rational functions) that any element of the vector space can be obtained as a finite sum of curves with single roots at the denominator. Our results give insight into the structure of these spaces, clarify the role of their polynomial and truly rational (non-polynomial) curves, and suggest applications to interpolation problems.
 \end{abstract}

\begin{keyword}
  Rational curve, polynomial curve, Pythagorean hodograph curve, partial fraction decomposition, Laurent series, canonical basis.
\end{keyword}

\end{frontmatter}

\section{Introduction}
\label{sec:introduction}

Polynomial and rational parametric curves are a traditional topic of Computer Aided Geometric Design and frequently appear in geometric modeling or numeric simulations. Some applications profit from rationality of the curve's unit tangent field, an observation that lead to the introduction of rational PH curves (``Pythagorean hodograph curves'') in \cite{farouki90c}. Ever since, PH curves have been closely investigated. Spatial PH curves were first considered in \cite{farouki94a} using a three polynomial preimage. A characterization of Pythagorean polynomial quadruples in terms of four polynomials was presented, in a different context, in \cite{Dietz}. Subsequently, two algebraic models for spatial PH curves were based on this characterization: The quaternion and Hopf map representations, as proposed in \cite{choi02b}. These forms are rotation invariant \cite{farouki02a} and serve as the foundation for many practical constructions of spatial PH curves and associated frames \cite{farouki08,faroukisurv}.

A vast majority of research papers focuses on \emph{polynomial} PH curves which have the advantage of a direct construction via integration of the hodograph. It was however noted in \cite{FaroukiSir2} that many properties of polynomial spatial PH curves can be transferred to the whole family of rational curves with the same tangent indicatrix. This family was first investigated in  \cite{FaroukiSir} where the envelope approach to planar rational PH curves of Pottmann and Farouki \cite{farouki96c} was generalized to rational PH space curves. While being rather straightforward and comprehensive, it seems that degree of the denominator polynomial and curve shape are difficult to control. The dual approach was continued in series of papers \cite{krajnc1,krajnc2,krajnc3} where the authors solve interpolation problems with rational spatial PH curves of low class. In \cite{FaroukiSakkalis2019} a special form of the rational hodograph is used to construct planar rational PH curves with rational arc-length function.

In the recent article \cite{kalkan22}, we have (together with co-authors) suggested a novel computation method for rational spatial PH curves. Motivated by the theory of motion polynomials in the dual quaternion model of space kinematics, c.f. \cite{husty12,hegedus13,li19} we identified rational PH curves with so called framing rational motions. Using suitable factorization we separated the spherical and translational part of the motion. The spherical component has been called Euler-Rodrigues motion or, equivalently referring to the images of a suitable orthogonal tripod, Euler-Rodrigues frame \cite{choi02,krajnc1}. Selecting appropriate parameters (the spherical motion component, represented by a quaternion polynomial $\mathcal{A} \in \H[t]$ and the PH curve's denominator polynomial $\alpha$) the  remaining coefficients can then be computed by solving a modestly sized and well-structured system of linear equations.

Here, we continue this approach. While \cite{kalkan22} was mostly concerned with the existence of truly rational (non-polynomial) solutions, we now investigate the complete structure of the solution spaces. While the formalism of motion polynomials yields a very natural and compact way of understanding this problem, it is not completely necessary and we abandoned it in this paper for two reasons. The first reason is to make our results better understandable for a wider audience. The second one is to make it obvious that our approach generalizes to more general cases of the tangent fields and also to higher dimensions. Our central result is Theorem~\ref{th:decomposition-multiple-roots} that, together with Theorems~\ref{th:dim-single-root}, \ref{th:canonical-decomposition}, and \ref{th:decomposition-curve}, allows us to construct canonical bases for solution spaces and also results in a rather surprising partial fraction decomposition for solution curves (Corollary~\ref{cor:PFD}).

We proceed with introducing some necessary basic information in Section~\ref{sec:preliminaries}. Our main results are presented in Section~\ref{sec:main-results}, at first for PH curves with a single root in the denominator (Section~\ref{sec:single-root}) and then for general denominators (Section~\ref{sec:general-denominators}). In order to make Section~\ref{sec:main-results} as clear as possible also several examples are included there. The corresponding technical proofs are collected in Section~\ref{sec:proofs} in order to allow for an uninterrupted flow of reading in Section~\ref{sec:main-results}. Our main tool of proof is a careful but rather technical study of an underlying system of linear equations in the special case of single root denominators. In concluding Section~\ref{sec:conclusion} we provide a more general understanding of our results as well as some possible directions for the future research.

\section{Preliminaries}
\label{sec:preliminaries}

In this section we will state the problem which we propose to solve. We will also place our approach in the context of existing constructions, in particular in the context of the theory of polynomial PH curves and of the two recent publications on rational PH curves \cite{FaroukiSir2,kalkan22}. Although the present paper is a continuation of the same subject, it is self-contained and can be read independently.

Let  us recall that skew field of quaternions $\H$ is the associative real algebra generated by the four basis elements $1$, $\qi$, $\qj$, $\qk$ together with the multiplication derived from the generating relations
\begin{equation*}
  \qi^2 = \qj^2 = \qk^2 = \qi\qj\qk = -1.
\end{equation*}
\emph{Conjugation} of a quaternion $\mathcal Q = q_0 + q_1 \qi + q_2 \qj + q_3 \qk \in \H$ is defined by changing signs of the coefficients at the complex units, i.e. $\Cj{\mathcal Q} = q_0 - q_1 \qi - q_2 \qj - q_3 \qk$, and the norm of $q$ is given by $\mathcal Q\Cj{\mathcal Q} = q_0^2 + q_1^2 + q_2^2 + q_3^2 \in \R$. Quaternion polynomials have been proven to be very useful in order to systematically construct \emph{polynomial} PH curves, see \cite{farouki08}. We propose to study the following general problem.

\begin{problem}
  \label{pr1}
  Given a quaternion valued polynomial $\mathcal{A}(t)$ determine the vector space of all the spatial rational curves $\mathbf{r}(t)$ having $\mathbf{F}(t) \coloneqq \mathcal{A}(t) \qi \Cj{\mathcal{A}}(t)$ for its tangent field, i.e. satisfying
  \begin{equation}
    \label{eq:1}
    \mathbf r'(t)\times {\mathbf F}(t)=0.
  \end{equation}
  Describe the structure of this solution space and, for ``interesting'' subspaces, construct canonical bases.
\end{problem}

As discussed more in detail in \cite{kalkan22}, the solutions to this problem for all input $\mathcal{A}(t)$ are precisely all the rational PH curves. In \cite{kalkan22} we were interested in criteria for the existence of rational (non-polynomial) solutions. In this article, we will provide a basis of the solution space and determine which basis elements are rational. Our approach is based on fixing not only $\mathcal{A}$ but also the denominator polynomial $\alpha$ and often also an upper bound for the degree of the numerator polynomial.

\begin{remark}
  \label{rem:F}
  It is possible to consider the polynomial ${\mathbf F}(t)$ as input to our problem. Since we will not make use of the PH property implied by ${\mathbf F}(t) = {\mathcal A}(t) \qi \Cj{\mathcal{A}}(t)$, all our results hold true for general ${\mathbf F}(t)$. The corresponding rational solution curves will have a prescribed polynomial direction field of tangents that not necessarily allows for a rational normalization.
\end{remark}

Note that rational tangent fields will give the same class of curves as the polynomial one. Indeed, multiplication by the common denominator does not change the tangent direction. Moreover, we can always assume that ${\mathbf F}(t)$ is free of polynomial factors. In the PH problem case this is equivalent to ${\mathcal A}(t)$ being free of polynomial factors and of right factors of the type $a(t)+\qi b(t)$ with real polynomials $a(t)$, $b(t)$, c.f.~\cite[Theorem 3.5, Definition 3.4]{kalkan22}. From now on we will thus assume that ${\mathbf F}(t)$ is free of real polynomial factors. We set $a \coloneqq \deg {\mathcal A}(t)$ whence the degree of ${\mathbf F}(t)$ equals~$2a$.

Polynomial PH curves are well studied in literature. In our context, they constitute an important building block of solution spaces that, fortunately, is easy to understand. We denote the vector space of polynomial solutions to \eqref{eq:1} of degree at most $M$ by $\mathcal{P}_{\mathcal A}^M$ and, for later reference, prove a simple and well-known lemma.

\begin{lemma}\label{lem:polPH}
  The space $\mathcal{P}_{{\mathcal A}}^M$ has dimension $M-2a+3$.
\end{lemma}
\begin{proof}
  Any element of $\mathcal{P}_{{\mathcal A}}^M$ has the form $\int \lambda(t){\mathbf F}(t)\,\mathrm{d}t$ where $\lambda$ is a real polynomial of degree at most $M-2a-1$. The free parameters are its $M-2a$ coefficients plus the vectorial constant of integration.
\end{proof}

The proof of above lemma demonstrates the straightforward approach to polynomial solution curves via integration. It is possible because polynomials are closed under integration --- a property not enjoyed by rational functions. So far, the predominant way of dealing with rational PH curves is the approach of \cite{FaroukiSir} that constructs rational PH curves as envelopes of their, likewise rational, sets of osculating planes. The simplest closed formula provided from this approach appeared in \cite{krajnc1, FaroukiSir2}. All the solutions to Problem~\ref{pr1} have the form
\begin{equation*}
  {\mathbf r}(t)=\frac{f(t)   \, \mathbf {u}'(t)\times \mathbf {u}''(t) +
    f'(t)  \, \mathbf {u}''(t)\times \mathbf {u}(t) +
    f''(t) \, \mathbf {u}(t)\times \mathbf {u}'(t)}
  {\det[\,\mathbf {u}(t),\mathbf {u}'(t),\mathbf {u}''(t)\,]},
\end{equation*}
where $\mathbf{u}(t)={\mathbf F}(t)\times {\mathbf F}'(t)$ and $f(t)$ is an arbitrary rational function. While this formula fully describes the system and results from some elegant geometric construction, it has several computational issues. Cusps of the curve occur in an uncontrollable way. The degree of the curve is difficult to control as well as the condition for its polynomiality. On the related subject, the denominator is a square
\begin{equation}
  \label{eq:2}
  \det[\,\mathbf {u}(t),\mathbf {u}'(t),\mathbf {u}''(t)\,]=\det[\,{\mathbf F}(t),{\mathbf F}'(t),{\mathbf F}''(t)\,]^2.
\end{equation}
One can in principle try to set the function $f$ so that one occurrence of the determinant $\det[\,\mathbf{F}(t),\mathbf{F}'(t),\mathbf{F}''(t)\,]$ is canceled. In all attempts however, also the second occurrence of the same factor was canceled. So it seems that the factors of the denominator of \eqref{eq:2} can be either of power three (coming from the denominator of $f''$) or of degree two coming from $\det[\,\mathbf{u}(t),\mathbf{u}'(t),\mathbf{u}''(t)\,]$ with generally no possibility of simple cancellation. This behavior, quite mysteriously appearing in \cite{FaroukiSir2}, will be to a great extent elucidated in the present paper.

From now on we usually omit the argument $t$ when writing polynomials. The following lemma summarizes the approach of \cite{kalkan22} and reduces the solution of Problem \ref{pr1} to the solution of a system of linear equations.

\begin{lemma}
  \label{lem:2}
  Consider a rational curve with parametric equation
  \begin{equation}
    \label{eq:3}
    \mathbf r = \frac{-2{\mathbf b}}{\alpha}
  \end{equation}
  where ${\mathbf b} = \sum_{i=0}^N {\mathbf b}_it^i$ is a vector valued polynomial (its coefficients are elements of $\R^3$) and $\alpha \in \R[t]$ is a real polynomial. Then $\mathbf{r}$ is a solution to Problem~\ref{pr1} if and only if
  \begin{equation}
    \label{eq:4}
    \alpha' {\mathbf b} - \alpha {\mathbf b}' = \mu \mathbf F
  \end{equation}
  for some $\mu \in \R[t]$.
\end{lemma}
\begin{proof}
  A detailed proof can be found in \cite[Theorem 3.6]{kalkan22}. It is based on the fact that $(\alpha' {\mathbf b} - \alpha {\mathbf b}')$ is a numerator of $\mathbf r'$. Also $\mu$ can be considered polynomial (and not rational) because $\mathbf F$ has no real factors.
\end{proof}

The constant $-2$ in \eqref{eq:3} is an artifact of the proof of Lemma~\ref{lem:2} in \cite{kalkan22} via dual quaternion polynomials. For the sake of consistency with \cite{kalkan22} we will not omit it, even if this would well be possible. Note that the representation \eqref{eq:3} is only unique up to multiplication of both, ${\mathbf b}$ and $\alpha$, with real polynomials and up to rational re-parametrizations. The former ambiguity is important for us in this paper.

With $\mathbf F$, $\alpha$, $\deg {\mathbf b}$, and $\deg \mu$ prescribed, expanding and comparing coefficients in $t$, Equation~\eqref{eq:4} leads to a system of homogeneous linear equations for computing the coefficients of ${\mathbf b}$ and $\mu$. Let us briefly discuss the relevance and meaning of the different polynomials involved in~\eqref{eq:4}:
\begin{itemize}
\item The solutions for ${\mathbf b}$ and the corresponding solution curves \eqref{eq:3} form a vector space. A careful discussion of the solution space for ${\mathbf b}$ will be an important tool of proof but we are actually interested in the rational curve $\mathbf r$. We will therefore formulate results in terms of $\mathbf r$ (Section~\ref{sec:main-results}) while ${\mathbf b}$ will appear in technical proofs (Section~\ref{sec:proofs}).
\item We often refer to $\alpha$ as the \emph{denominator polynomial.} This includes the possibility that \eqref{eq:3} is not reduced, that is, the resulting rational curve $\mathbf r$ may have representations with denominators of lower degree. Even polynomial solution curves $\mathbf r$ are possible and will occur if $\alpha$ is a factor of~${\mathbf b}$.
\item The polynomials $\mathbf F$ and $\alpha$ are considered as input to our problem. The polynomials $\mathcal A$ and $\mathbf F = \mathcal A \qi \Cj{\mathcal A}$ will be considered as fixed throughout the text while we will allow for different values of the denominator polynomial $\alpha$, sometimes even within the same lemma, theorem, or example.
\item The polynomial $\mu$ is not of direct interest to us. We are rather interested in ${\mathbf b}$ and, ultimately, in $\mathbf r$. However, the relevance of $\mu$ should not be underestimated. It directly influences $\mathbf r$ and, together with $\mathcal A$, encodes the parametric speed of the resulting curve. Therefore, it plays an important role, for example in interpolation problems with rational PH curves.
\end{itemize}

In Section~\ref{sec:single-root} we will derive several results pertaining to rational PH curves with a single root $\beta$ in the denominator, that is $\alpha = (t-\beta)^n$ for some $n \in \mathbb{N}$. These equally apply to real or complex values of $\beta$. The reasons for admitting complex roots is Section~\ref{sec:general-denominators} where we formulate results for arbitrary denominator polynomials $\alpha$ in terms of its complete factorization over~$\C$.

Throughout this paper, we will often make a genericity assumptions on the input data:

\begin{definition}
  \label{def:generic}
  The data $\mathcal A \in \H[t]$, $\alpha \in \R[t]$ is called \emph{generic} if any three consecutive coefficients ${\mathbf f}_{i-1}$, ${\mathbf f}_i$, ${\mathbf f}_{i+1}$ in the Taylor expansion $\mathbf{F} = \mathcal A \qi \Cj{\mathcal A} = \sum_{i=0}^{2a} {\mathbf f}_i (t-\beta)^i$ at every root $\beta \in \C$ of $\alpha$ are linearly independent. Otherwise the data are called \emph{non-generic}.
\end{definition}

By the central result of \cite{kalkan22}, the linear independence of ${\mathbf f}_{0}$, ${\mathbf f}_1$, ${\mathbf f}_{2}$ for all the roots of $\alpha$ implies the existence of non-polynomial solutions having a root of multiplicity at least three. In case of non-generic input data, variants of most of our results still hold true. We will occasionally hint at these possibilities.

\section{Main Results}
\label{sec:main-results}

We will describe the complete space of rational and polynomial PH curves with a prescribed tangent indicatrix. First we will analyse in detail the case where the denominator $\alpha$ has only one (possibly multiple) root $\beta$. Then we provide decomposition results in order to handle general denominators. Proofs of lemmas in this section are collected in Section~\ref{sec:proofs}.

\subsection{Denominators with a Single Root}
\label{sec:single-root}

Let us suppose through this subsection that the denominator $\alpha$ is of the shape $\alpha = (t-\beta)^n$ for some $\beta \in \C$ and $n \in \mathbb{N}$. The main reason to consider the complex number field is to be able to later handle irreducible (over $\mathbb R$) quadratic factors. The choice of a non-real $\beta$ does not lead to any fundamental difference. More precisely the solutions of the following system of equations will be real for $\beta$ real and complex for $\beta$ complex.

In accordance with the previous notation we denote by $\mathcal{R}_{{\mathcal A},\beta}$ the linear space of all rational PH curves with the tangent indicatrix $\mathcal A\qi\Cj{\mathcal A}/(\mathcal A\Cj{\mathcal A})$ which have at most one root $\beta$ at the denominator. Any element $\mathbf{r} \in \mathcal{R}_{\mathcal A,\beta}$ has a unique Laurent expansion \cite[Section~8.1]{howie03}
\begin{equation}
  \label{eq:5}
  {\mathbf r}= \sum_{i=-\infty}^\infty {\mathbf r}_i (t-\beta)^i,
\end{equation}
with only finitely many nonzero coefficients ${\mathbf r}_i \in \R^3$. For any $m$, $M \in \Z$ with $m \le M$ we denote by $\mathcal{R}_{{\mathcal A},\beta}^{m,M}$ the linear subspace of PH curves whose Laurent series representation \eqref{eq:5} satisfies ${\mathbf r}_i = 0$ if $i < m$ or if $i > M$.

Clearly these spaces are nested for decreasing $m$ or increasing $M$. Note also, that many of these spaces may be trivial (consisting only of the zero constant curve). For $m\ge 0$ the space $\mathcal{R}_{{\mathcal A},\beta}^{m,M}$ contains only polynomial curves and does not depend on $\beta$. In particular the space $\mathcal{R}_{{\mathcal A},\beta}^{0,0}$ has dimension $3$ and contains only the constant curves, that is $\mathcal{R}_{{\mathcal A},\beta}^{0,0} = \R^3$. This very simple space is important for understanding the structure of all spaces $\mathcal{R}_{{\mathcal A},\beta}^{m,M}$. As we shall see later, the value $M=0$ is the only exception to the formula $\dim \mathcal{R}_{{\mathcal A},\beta}^{m,M} \le \dim \mathcal{R}_{{\mathcal A},\beta}^{m,M-1} + 1$. For this reason we introduce a curve normalization defining
\begin{equation}
  \label{eq:6}
  {\mathcal Q}_{{\mathcal A},\beta} \coloneqq \{\mathbf{r} \in \mathcal{R}_{{\mathcal A},\beta} \mid {\mathbf r}_0 = 0\}
\end{equation}
and its Laurent cuts ${\mathcal Q}_{{\mathcal A},\beta}^{m,M} \coloneqq \{\mathbf{r} \in \mathcal{R}_{{\mathcal A},\beta}^{m,M} \mid {\mathbf r}_0 = 0\}$. This simply removes the translational degree of freedom from solution curves and allows us to formulate

\begin{lemma}
  \label{lem:M0-existence}
  Given $m \in \Z$, there is a minimal $M_0(m) \ge m$ such that $\dim {\mathcal Q}_{{\mathcal A},\beta}^{m,M_0(m)} = 1$ and $\dim {\mathcal Q}_{{\mathcal A},\beta}^{m,M} = 0$ for all $M < M_0(m)$.
\end{lemma}

For the generic case (c.f. Definition~\ref{def:generic}), the next lemma gives the precise value of $M_0(m)$. For non-generic cases, similar formulas can be derived.

\begin{lemma}
  \label{lem:M0}
  For generic $\mathcal{A}$, $\beta$ we have
  \begin{equation}
    \label{eq:7}
    M_0(m) =
    \begin{cases}
      2a + m     & \text{if $m < -2a$,}          \\
      2a + m + 3 & \text{if $-2 a \le m \le -2$,} \\
      2a + m + 2 & \text{if $m = -1$,}           \\
      2a + m + 1 & \text{if $m = 0$,}            \\
      2a + m     & \text{if $m \ge 1$.}
    \end{cases}
  \end{equation}
\end{lemma}

In order to obtain a better idea about the behavior of $M_0(m)$ let us formulate

\begin{corollary}
  For generic $\mathcal{A}$, $\beta$ we have
  \begin{equation*}
    M_0(-2) = M_0(-1) = M_0(0) = M_0(1) = 2a + 1.
  \end{equation*}
  The restriction of $M_0$ to the set $\mathbb Z \setminus \{-2,-1,0\}$ is strictly increasing (Figure~\ref{fig:M0}).
\end{corollary}

\begin{figure}
  \centering
\includegraphics{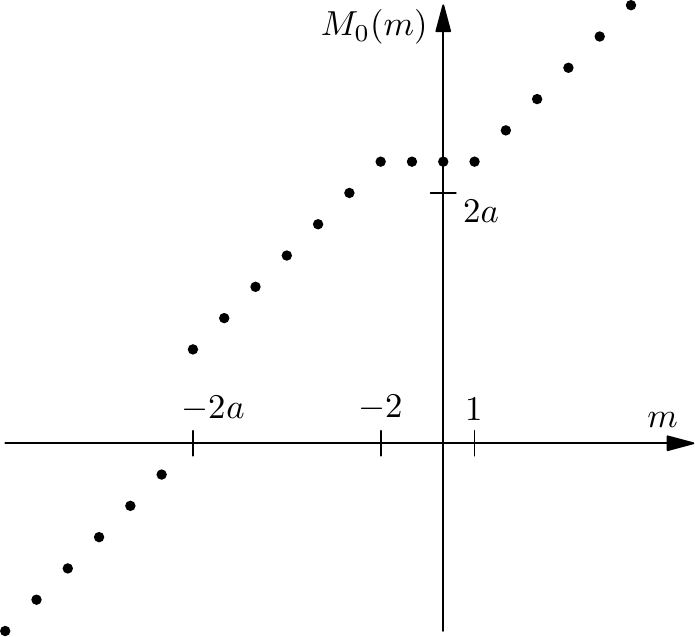}
  \caption{The function $M_0(m)$ for $a = 4$.}
  \label{fig:M0}
\end{figure}

\begin{lemma}
  \label{lem:qbasis}
  For generic $\mathcal{A}$, $\beta$ we have
  \begin{equation*}
    {\mathcal Q}_{{\mathcal A},\beta}^{-2,2a+1} = {\mathcal Q}_{{\mathcal A},\beta}^{-1,2a+1} =
    {\mathcal Q}_{{\mathcal A},\beta}^{0,2a+1} = {\mathcal Q}_{{\mathcal A},\beta}^{1,2a+1}.
  \end{equation*}
  Consequently, there are no normalized PH curves whose Laurent cut starts with a non-zero coefficient at the power $-2$, $-1$, or $0$. On the other hand, for $m\in \mathbb Z \setminus \{-2,-1,0\}$ there is a vector ${\mathbf q}_{{\mathcal A},\beta}^{m}\in {\mathcal Q}_{{\mathcal A},\beta}^{m,M_0(m)}$ with the non-zero lowest coefficient ${\mathbf r}_m={\mathbf F}(\beta)={\mathcal A}(\beta)\qi\Cj{\mathcal{A}}(\beta)$.
\end{lemma}

The observation for the lowest non-zero coefficient in Lemma~\ref{lem:qbasis} can be extended to any PH curve. Taking into account the case ${\mathbf r}_m=0$ which is in particular always true for $m \in \{-2, -1, 0\}$ we can formulate

\begin{lemma}
  \label{lem:lowCoef}
  Let ${\mathbf q}=\sum_{i=m}^M {\mathbf r}_i(t-\beta)^i\in {\mathcal Q}_{{\mathcal A},\beta}^{m,M}$. Then the coefficient ${\mathbf r}_m$ is a multiple of ${\mathbf F}(\beta)$.
\end{lemma}

The previous lemmas allow us to construct bases of ${\mathcal Q}_{{\mathcal A},\beta}$, $\mathcal{R}_{{\mathcal A},\beta}$, ${\mathcal Q}_{{\mathcal A},\beta}^{m,M}$ and $\mathcal{R}_{{\mathcal A},\beta}^{m,M}$.

\begin{theorem}
  \label{th:dim-single-root}
  Let $\mathbb{I} \coloneqq \{-2, -1, 0\}$ and assume that $\mathcal A$, $\beta$ are generic.
  \begin{enumerate}
  \item The space of normalized solutions is the direct sum of one-dimensional linear spaces, c.f. Figure~\ref{fig:1cuts}:
    \begin{equation}\label{eq:8}
      {\mathcal Q}_{{\mathcal A},\beta} = \bigoplus_{m \notin \mathbb{I}} {\mathcal Q}_{{\mathcal A},\beta}^{m,M_0(m)}.
    \end{equation}
  \item A subset of this list of one-dimensional spaces gives the space of normalized Laurent cuts:
    \begin{equation}\label{eq:9}
      {\mathcal Q}_{{\mathcal A},\beta}^{m,M} = \bigoplus_{\ell \in \mathbb{L}}
      {\mathcal Q}_{{\mathcal A},\beta}^{\ell,M_0(\ell)},\quad
      \mathbb{L} = \{ \ell \in \Z \mid m
      \le \ell, M_0(\ell) \le M \} \setminus \mathbb{I}.
    \end{equation}
  \item The non-normalized solutions are obtained by adding translations:
    \begin{equation}
      \label{eq:10}
      \mathcal{R}_{{\mathcal A},\beta} = \mathbb{R}^3 \oplus {\mathcal Q}_{{\mathcal A},\beta}.
    \end{equation}
  \item The same is true for the Laurent cuts but we have to distinguish whether the constant coefficient belongs to the Laurent cut or not:
    \begin{equation}
      \label{eq:11}
      \mathcal{R}_{{\mathcal A},\beta}^{m,M} =
      \begin{cases}
        {\mathcal Q}_{{\mathcal A},\beta}^{m,M} \oplus \mathbb{R}^3 & \text{ if $m \le 0 \le M$},\\
        {\mathcal Q}_{{\mathcal A},\beta}^{m,M} & \text{otherwise.}
      \end{cases}
    \end{equation}
  \end{enumerate}
\end{theorem}

\begin{proof}
  Equation~\eqref{eq:8} is equivalent to the fact that any element of ${\mathcal Q}_{{\mathcal A},\beta}$ can be expressed as a unique linear combination of elements ${\mathbf q}^{m}_{{\mathcal A},\beta}$, $m\in \mathbb Z \setminus \mathbb I$ as defined in Lemma~\ref{lem:qbasis}; Equation~\eqref{eq:9} merely adds the restriction $m\in \mathbb L$.

  Suppose that ${\mathbf q}=\sum_{i=m}^M {\mathbf r}_i(t-\beta)^i\in {\mathcal Q}_{{\mathcal A},\beta}^{m,M}$ and let us prove existence of linear combinations by induction on $(M-m)$. If $m=M$ only zero PH curves ${\mathbf q}$ exist in ${\mathcal Q}_{{\mathcal A},\beta}^{m,M}$ and the linear combination to obtain ${\mathbf q}$ is trivial. Suppose now that $m < M$ and ${\mathbf q}$ is non-zero. By Lemma~\ref{lem:M0-existence} $M \ge M_0(m)$ and, by Lemma~\ref{lem:lowCoef}, the coefficient ${\mathbf r}_m$ is a multiple of ${\mathbf F}(\beta)$. Consequently there exists $\lambda \in \mathbb C$ so that $({\mathbf q}-\lambda {\mathbf q}_{{\mathcal A},\beta}^m) \in {\mathcal Q}_{{\mathcal A},\beta}^{m+1,M}$. This is the induction step and existence of suitable linear combinations is proved.

  The uniqueness of the linear combination follows from the obvious linear independence of the set of basis curves $\{{\mathbf q}_{{\mathcal A}, \beta}^m\}_{m\in \mathbb Z\setminus \mathbb L}$. Equations~\eqref{eq:10} and \eqref{eq:11} are direct consequences of \eqref{eq:8} and \eqref{eq:9}, respectively.
\end{proof}

\begin{figure}
  \centering
\includegraphics[width=12cm]{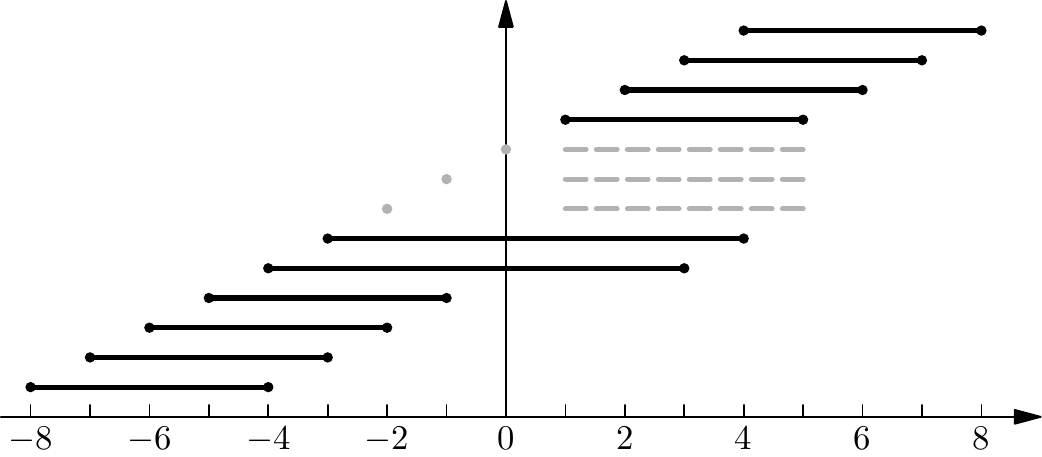}
  \caption{The lines represent intervals $[m, M_0(m)]$ for the decomposition
    \eqref{eq:8} of ${\mathcal Q}_{\mathcal A, \beta}$ into one-dimensional
    space for generic $\alpha$ and $\deg \mathcal A=2$. The dots represent the
    ``missing'' spaces to $m \in \{-2,-1,0\}$, the dashed lines represent
    their smallest Laurent cuts.}
  \label{fig:1cuts}
\end{figure}

\begin{remark}
  For non-generic cases analogous results to Lemma~\ref{lem:M0}, Lemma~\ref{lem:qbasis} and Theorem~\ref{th:dim-single-root} hold and in any concrete case can be easily determined. Their complete enumeration is rather complicated and will be omitted. Let us point out, however, the two essential differences:
  \begin{itemize}
  \item For non-generic cases the function $M_0(m)$ changes and some of its values may be lower than in the generic case.
  \item In particular in the case of linearly dependent ${\mathbf f}_{0}$, ${\mathbf f}_1$, ${\mathbf f}_{2}$, the space identities of Lemma \ref{lem:qbasis} may not hold and the index set $\mathbb{I}$ is to be replaced by $\{-1,0\}$ or even~$\{0\}$.
  \end{itemize}
\end{remark}

Now we illustrate Theorem~\ref{th:dim-single-root} an example. The computation
of basis vectors is straightforward: Solve the linear system resulting from
\eqref{eq:4} for $\alpha = (t-\beta)^m$, ${\mathbf b} = \sum_{i=0}^N
\mathbf{b}_i (t-\beta)^i $ where $N = m + M_0(m)$. In this way we obtain the
minimal number of non-vanishing Laurent coefficients. Our results so far
guarantee a one-dimensional solution space. The ambiguity is resolved by having
$\mathbf b_0 = -2\mathbf{F}(\beta)$, c.f.~\eqref{eq:3}.

\begin{example}
  \label{ex:single-root}
  We consider the example given by
  \begin{equation*}
    \mathcal{A} = (7 t^2-22 t+10) + (-19 t^2+14 t)\qi + (-26 t^2+16 t)\qj + (-2t^2+12 t)\qk.
  \end{equation*}
  and $\beta = -10$ which already appeared in other places, c.f., \cite[Example~5.5]{kalkan22} or \cite[Example~1]{FaroukiSir2} where it is discussed in the context of rational PH curves with rational rotation minimizing motions. It can easily be confirmed that the data $\mathcal{A}$, $\beta$ is generic.

  Using Lemma~\ref{lem:M0} with $a = \deg\mathcal{A} = 2$ we compute $M_0(m)$ for some values of $m$:
  \begin{center}
    \begin{tabular}{c|rrrrrrrrrrr}
      $m$      & $-7$ & $-6$ & $-5$ & $-4$ & $-3$ & $-2$ & $-1$ & $0$ & $1$ & $2$ & $3$\\
      \hline
      $M_0(m)$ & $-3$ & $-2$ & $-1$ & $3$  & $4$  & $5$  & $5$  & $5$ & $5$ & $6$ & $7$
    \end{tabular}
  \end{center}
  Each space ${\mathcal Q}_{{\mathcal A},\beta}^{m,M_{0}}$ is of dimension one. We present basis vectors for some of these spaces that are normalized as suggested by Lemma~\ref{lem:qbasis}:
  \begin{multline*}
    {\mathbf q}_{{\mathcal A},-10}^{-5} =
    \begin{psmallmatrix}-3298104\\9815520\\6340640\end{psmallmatrix}(t+10)^{-5}+\begin{psmallmatrix}1571650\\-4898740\\-2913180\end{psmallmatrix}(t+10)^{-4}+\begin{psmallmatrix}-897340/3\\975440\\1599440/3\end{psmallmatrix}(t+10)^{-3}\\+\begin{psmallmatrix}28430\\-96880\\-48560\end{psmallmatrix}(t+10)^{-2}+\begin{psmallmatrix}-1350\\4800\\2200\end{psmallmatrix}(t+10)^{-1}     \in {\mathcal Q}_{{\mathcal A},-10}^{-5,-1},
  \end{multline*}
  \begin{multline*}
    {\mathbf q}_{{\mathcal A},-10}^{-4} =
    \begin{psmallmatrix}-3298104\\9815520\\6340640\end{psmallmatrix}(t+10)^{-4}\\+\begin{psmallmatrix}1007210705356/779325\\-636312533296/155865\\-1107790524016/467595\end{psmallmatrix}(t+10)^{-3}+\begin{psmallmatrix}-6694077029408/36628275\\22536453169244/36628275\\11582106201908/36628275\end{psmallmatrix}(t+10)^{-2}\\+\begin{psmallmatrix}406329786974/36628275\\-1444668933424/36628275\\-220650209776/12209425\end{psmallmatrix}(t+10)^{-1}+\begin{psmallmatrix}2069307793/36628275\\-246327880/1465131\\-795217132/7325655\end{psmallmatrix}(t+10)\\+\begin{psmallmatrix}-346728479/73256550\\571515752/36628275\\306325034/36628275\end{psmallmatrix}(t+10)^{2}+\begin{psmallmatrix}320307/2441885\\-3416608/7325655\\-4697836/21976965\end{psmallmatrix}(t+10)^{3}     \in {\mathcal Q}_{{\mathcal A},-10}^{-4,3},
  \end{multline*}
  \begin{multline*}
    {\mathbf q}_{{\mathcal A},-10}^{1} =
    \begin{psmallmatrix}-3298104\\9815520\\6340640\end{psmallmatrix}(t+10)+\begin{psmallmatrix}628660\\-1959496\\-1165272\end{psmallmatrix}(t+10)^{2}+\begin{psmallmatrix}-179468/3\\195088\\319888/3\end{psmallmatrix}(t+10)^{3}\\+\begin{psmallmatrix}2843\\-9688\\-4856\end{psmallmatrix}(t+10)^{4}+\begin{psmallmatrix}-54\\192\\88\end{psmallmatrix}(t+10)^{5}     \in {\mathcal Q}_{{\mathcal A},-10}^{1,5},
  \end{multline*}
  \begin{multline*}
    {\mathbf q}_{{\mathcal A},-10}^{2} =
    \begin{psmallmatrix}-3298104\\9815520\\6340640\end{psmallmatrix}(t+10)^{2}+\begin{psmallmatrix}2514640/3\\-7837984/3\\-1553696\end{psmallmatrix}(t+10)^{3}+\begin{psmallmatrix}-89734\\292632\\159944\end{psmallmatrix}(t+10)^{4}\\+\begin{psmallmatrix}22744/5\\-77504/5\\-38848/5\end{psmallmatrix}(t+10)^{5}+\begin{psmallmatrix}-90\\320\\440/3\end{psmallmatrix}(t+10)^{6}     \in {\mathcal Q}_{{\mathcal A},-10}^{2,6}.
  \end{multline*}
  The spaces ${\mathcal Q}_{{\mathcal A},-10}^{-5,-1}$ and ${\mathcal Q}_{{\mathcal A},-10}^{-4,3}$ contain rational curves, while ${\mathcal Q}_{{\mathcal A},-10}^{1,5}$ ($= {\mathcal Q}_{{\mathcal A},-10}^{0,5} = {\mathcal Q}_{{\mathcal A},-10}^{-1,5} = {\mathcal Q}_{{\mathcal A},-10}^{-2,5}$) is polynomial as well as ${\mathcal Q}_{{\mathcal A},-10}^{2,6}$. Note the absence of the constant coefficients in ${\mathcal Q}_{{\mathcal A},-10}^{-4,3}$ and the absence of the constant coefficients and the coefficients to $(t+10)$ in ${\mathcal Q}_{{\mathcal A},-10}^{2,6}$. By Theorem~\ref{th:dim-single-root}, Items~2 and 4, we have
  \begin{equation*}
    \mathcal{R}_{{\mathcal A},-10}^{-7,-1} = {\mathcal Q}_{{\mathcal A},-10}^{-7,-1} =
    {\mathcal Q}_{{\mathcal A},-10}^{-7,-3} \oplus
    {\mathcal Q}_{{\mathcal A},-10}^{-6,-2} \oplus
    {\mathcal Q}_{{\mathcal A},-10}^{-5,-1}
  \end{equation*}
  and
  \begin{equation*}
    \mathcal{R}_{{\mathcal A},-10}^{-5,5} = {\mathcal Q}_{{\mathcal A},-10}^{-5,5} \oplus \R^3 =
    {\mathcal Q}_{{\mathcal A},-10}^{-5,-1} \oplus
    {\mathcal Q}_{{\mathcal A},-10}^{-4,3} \oplus
    {\mathcal Q}_{{\mathcal A},-10}^{-3,4} \oplus
    {\mathcal Q}_{{\mathcal A},-10}^{1,5} \oplus
    \R^3.
  \end{equation*}
\end{example}

Referring to above example, we note that ${\mathcal Q}_{{\mathcal A},-10}^{1,5} \oplus \R^3$ is just the space of polynomial PH curves of degree at most five, a space which is independent of the concrete value for $\beta$. Moreover, the remaining part ${\mathcal Q}_{{\mathcal A},-10}^{-5,-1} \oplus {\mathcal Q}_{{\mathcal A},-10}^{-4,3} \oplus {\mathcal Q}_{{\mathcal A},-10}^{-3,4}$ in the additive decomposition of $\mathcal{R}_{{\mathcal A},-10}^{-5,5}$ does not contain any polynomial PH curves. These observations generalize:

\begin{definition}
  \label{def:Xspace}
  For negative $m$ we define the spaces of purely rational PH curves which do not contain any polynomial PH curves:
  \begin{equation*}
    \mathcal{X}_{{\mathcal A},\beta} \coloneqq \bigoplus_{i=-\infty}^{-3} {\mathcal Q}_{{\mathcal A},\beta}^{i,M_0(i)}
    \quad\text{and}\quad
    \mathcal{X}_{{\mathcal A},\beta}^{m,M} \coloneqq \bigoplus_{i \in \mathbb{L}} {\mathcal Q}_{{\mathcal A},\beta}^{i,M_0(i)},
\end{equation*}
  where
  $\mathbb{L} = \{ i \in \Z \mid -3>i \ge m, M_0(i) \le M\}$.
\end{definition}

\begin{theorem}
  \label{th:canonical-decomposition}
  For any $M \in \mathbb{N}_0$, $m < 0$ and $\beta \in \C$ we have $\mathcal{R}_{{\mathcal A},\beta}^{0,M} =\mathcal{P}_{{\mathcal A}}^M$ (the space of polynomial PH curves of degree at most $M$),
  \begin{equation}
    \label{eq:12}
    \mathcal{R}_{{\mathcal A},\beta} =  \mathcal{X}_{{\mathcal A},\beta} \oplus \mathcal{P}_{{\mathcal A}} \quad \text{and}\quad \mathcal{R}_{{\mathcal A},\beta}^{m,M}=\mathcal{X}_{{\mathcal A},\beta}^{m,M}   \oplus \mathcal{P}_{{\mathcal A}}^M.
  \end{equation}
\end{theorem}

\begin{proof}
  The space $\mathcal{R}_{{\mathcal A},\beta}^{0,M}$ contains all solutions whose Laurent cut can be written as a Taylor polynomial of degree at most $M$ at $t = \beta$. It obviously equals the space $\mathcal{P}_{\mathcal A}^M$ and is independent of~$\beta$.

  By Theorem~\ref{th:dim-single-root} we have
  \begin{equation}
    \label{eq:13}
    \mathcal{R}_{{\mathcal A},\beta} = \bigoplus_{m \notin \mathbb{I}} {\mathcal Q}_{{\mathcal A},\beta}^{m,M_0(m)}\oplus \mathbb R^3
    = \underbrace{\biggl( \bigoplus_{i=-\infty}^{-3} {\mathcal Q}_{{\mathcal A},\beta}^{i,M_0(i)} \biggr)}_{\mathcal{X}_{{\mathcal A},\beta}} \oplus \underbrace{\biggl( \bigoplus_{i=1}^\infty {\mathcal Q}_{{\mathcal A},\beta}^{i,M_0(i)} \biggr) \oplus \R^3}_{\mathcal{P}_{\mathcal A}}.
  \end{equation}
  Note that the sum is really direct because $\mathcal{X}_{{\mathcal A},\beta}$ does not contain any polynomial PH curves due to $M_0(-3)=2a$ and by Lemma~\ref{lem:polPH} there are no non-constant polynomial PH curves of degree less than $2a+1$. The remaining equality in \eqref{eq:12} can be seen similarly.
\end{proof}

Theorem~\ref{th:canonical-decomposition} provides a canonical decomposition of $\mathcal{R}_{{\mathcal A},\beta}^{m,M}$ into spaces of non-polynomial and of polynomial PH curves which we will often use in the next section. In case of $M < 0$, the space $\mathcal{P}_{{\mathcal A}}^{M}$ only consists of the zero solution, in case of $0 \le M \le 2a$, it equals the space $\R^3$ of constant solutions.

\subsection{General Denominators}
\label{sec:general-denominators}
We are now ready to relate the spaces of PH curves for different denominator roots and show that any PH curve admits a direct additive decomposition into PH curves with single root denominators. Suppose that a rational curve ${\mathbf r}(t)$ has several roots in the denominator, one of them being $\beta$. Then it can still be written using its Laurent expansion \eqref{eq:5} but with the difference that infinitely many positive power coefficients are non-zero. The lowest non-zero coefficient ${\mathbf r}_m$ is determined by the multiplicity $-m$ of the root $\beta$. The radius of convergence of the infinite sum is given by the distance of the closest other root.

The following lemma is essential for decomposing such a curve into rational curves with single denominator roots.

\begin{lemma}
  \label{lem:finL}
  For a given $\mathcal{A} \in \H[t]$ let $\mathbf r\in \mathcal{R}_{{\mathcal A}}$ be a PH curve and assume that $\beta \in \mathbb C$ is a root of its denominator with multiplicity $n>0$. Then $\mathbf r$ has a Laurent expansion of the shape
  \begin{equation}
    \label{eq:14}
    {\mathbf r}(t) = \sum_{i=m}^\infty {\mathbf r}_i (t-\beta)^i
  \end{equation}
  with $m = -n$ and there exists a PH curve $\tilde {\mathbf r}(t) \in \mathcal{R}_{{\mathcal A},\beta}$ with finite Laurent expansion
  \begin{equation}
    \label{eq:15}
    \tilde {\mathbf r}(t) = \sum_{i=m}^{2a} \tilde {\mathbf r}_i (t-\beta)^i
  \end{equation}
  so that $\tilde {\mathbf r}_i={\mathbf r}_i$ for $i=m, m-1, \ldots, 0$.
\end{lemma}

\begin{theorem}
  \label{th:decomposition-multiple-roots}
  For $i \in \{1,2,\ldots,k\}$ let $\beta_i \in \mathbb C$ be mutually distinct numbers, $m_i \in \mathbb Z_0^-$ and $M_i \in \mathbb Z$ arbitrary. Then
  \begin{equation*}
    \bigcap_{i=1}^k \mathcal{R}_{{\mathcal A},\beta_i}^{m_i,M_i} =
    \mathcal{P}_{{\mathcal A}}^{\min\{M_1,M_2,\ldots,M_k\}}
  \end{equation*}
  and
  \begin{equation*}
    \sum_{i=1}^k \mathcal{R}_{{\mathcal A},\beta_i}^{m_i,M_i} =
    \biggl( \bigoplus \mathcal{X}_{{\mathcal A},\beta_i}^{m_i,M_i} \biggr) \oplus \mathcal{P}_{{\mathcal A}}^{\max\{M_1,M_2,\ldots,M_k\}}.
  \end{equation*}
\end{theorem}

\begin{proof}
  Let $\beta_1\neq \beta_2\in \mathbb
  C$ and $m_1$, $m_2 \in \mathbb Z_0^-$. Then for any $M_1$, $M_2$
  \begin{equation*}
    \mathcal{R}_{{\mathcal A},\beta_1}^{m_1,M_1} \cap \mathcal{R}_{{\mathcal A},\beta_2}^{m_2,M_2} =
    (\mathcal{X}_{{\mathcal A},\beta_1}^{m_1,M_1} \oplus \mathcal{P}_{{\mathcal A}}^{M_1})\cap (\mathcal{X}_{{\mathcal A},\beta_2}^{m_2,M_2} \oplus \mathcal{P}_{{\mathcal A}}^{M_2})
    = \mathcal{P}^{\min\{M_1,M_2\}}
  \end{equation*}
  because the intersection of $\mathcal{X}_{{\mathcal A},\beta_1}^{m_1,M_1}$ and
  $\mathcal{X}_{{\mathcal A},\beta_2}^{m_2,M_2}$ contains only $0$. Moreover,
  \begin{equation*}
    \mathcal{R}_{{\mathcal A},\beta_1}^{m_1,M_1} + \mathcal{R}_{{\mathcal A},\beta_2}^{m_2,M_2} =
    \mathcal{X}_{{\mathcal A},\beta_1}^{m_1,M_1} \oplus \mathcal{X}_{{\mathcal A},\beta_2}^{m_2,M_2} \oplus \mathcal{P}_{{\mathcal A}}^{\max\{M_1,M_2\}}.
  \end{equation*}
  This proves the theorem for $k = 2$. For general $k$, an inductive proof along similar lines is possible.
\end{proof}

\begin{theorem}
  \label{th:decomposition-curve}
  Given a rational PH curve $\mathbf r = \frac{-2{\mathbf b}}{\alpha}$ with a vectorial polynomial ${\mathbf b}$ and a polynomial $\alpha = \prod_{i=1}^k (t-\beta_i)^{n_i}$ we have
  \begin{equation*}
    \mathbf r \in \biggl( \bigoplus_{i=1}^k \mathcal{X}_{{\mathcal A},\beta_i}^{m_i,2a} \biggr) \oplus \mathcal{P}_{{\mathcal A}}^{\deg {\mathbf b} - \deg \alpha}
  \end{equation*}
  where $m_i = -n_i$.
\end{theorem}

\begin{proof}
  By Lemma \ref{lem:finL} there exists for each $i=1, 2, \ldots, k$ a PH curve $\mathbf {\mathbf s}_i\in \mathcal{X}_{{\mathcal A},\beta_i}^{m_i,2a}$ having the same negative index Laurent coefficients as $\mathbf{r}$ with respect to $(t-\beta_i)$. For this reason
  \begin{equation}
    \label{eq:16}
    \mathbf p:=\mathbf r-(\mathbf s_1+\ldots+\mathbf s_k)
  \end{equation}
  has no singular values and therefore is polynomial. In order to determine its degree, let us put the right-hand side of \eqref{eq:16} on the common denominator $\alpha$ and consider its numerator. The components $\mathbf s_i$ will contribute by the degree at most $2a+\deg \alpha$ and $\mathbf r$ by the degree $\deg {\mathbf b}$. If $\deg {\mathbf b} \le 2a + \deg \alpha$ then $\mathbf p$ is of degree at most $2a$ and must be constant due to Lemma \ref{lem:polPH}. If on the other hand $\deg {\mathbf b} > 2a + \deg \alpha$ the degree of $\mathbf p$ is at most $(\deg {\mathbf b} - \deg \alpha)$.
\end{proof}

We illustrate Theorem~\ref{th:decomposition-curve} at hand of an example that also shows how to handle irreducible (over $\R$) factors of $\alpha$ if one wishes to avoid complex numbers.

\begin{example}
  \label{ex:decomposition-curve}
  We consider rational PH curve
  \begin{multline*}
    {\mathbf r}= \frac{1}{\alpha}
    \Bigl(
    \begin{psmallmatrix} -0.0510 \\ -0.0004 \\ -0.0050 \end{psmallmatrix} +
\begin{psmallmatrix} -0.1000 \\ -0.0017 \\ -0.0200 \end{psmallmatrix}t +
\begin{psmallmatrix} -0.1000 \\ -0.0017 \\ -0.0783 \end{psmallmatrix}t^{2} +
\begin{psmallmatrix} -0.1000 \\ 0.0056 \\ -0.1783 \end{psmallmatrix}t^{3} +
\begin{psmallmatrix} -0.1000 \\ 0.0297 \\ -0.2758 \end{psmallmatrix}t^{4}\\ +
\begin{psmallmatrix} -0.1000 \\ 0.0849 \\ -0.2620 \end{psmallmatrix}t^{5} +
\begin{psmallmatrix} -0.1000 \\ 0.2006 \\ -0.1327 \end{psmallmatrix}t^{6} +
\begin{psmallmatrix} -31.4912 \\ 0.4749 \\ 0.0822 \end{psmallmatrix}t^{7} +
\begin{psmallmatrix} -13.7803 \\ 0.0853 \\ 18.0038 \end{psmallmatrix}t^{8} +
\begin{psmallmatrix} -17.9690 \\ -1.1335 \\ 0.2412 \end{psmallmatrix}t^{9} +
\begin{psmallmatrix} -19.9080 \\ -0.9355 \\ 3.0462 \end{psmallmatrix}t^{10}\\ +
\begin{psmallmatrix} -6.7257 \\ -1.2111 \\ -9.9858 \end{psmallmatrix}t^{11} +
\begin{psmallmatrix} -25.0542 \\ -0.2798 \\ -18.3100 \end{psmallmatrix}t^{12} +
\begin{psmallmatrix} -23.4789 \\ 0.5064 \\ -9.5939 \end{psmallmatrix}t^{13} +
\begin{psmallmatrix} -16.7710 \\ 0.5126 \\ -7.7792 \end{psmallmatrix}t^{14} +
\begin{psmallmatrix} -14.4895 \\ 0.5456 \\ -2.5465 \end{psmallmatrix}t^{15} +
\begin{psmallmatrix} -4.5060 \\ 0.2670 \\ 1.4686 \end{psmallmatrix}t^{16}
     \Bigr)
  \end{multline*}
  to the polynomial $\mathcal{A}$ of Example~\ref{ex:single-root} and the denominator polynomial $\alpha = (t+1)^4(t^2+1)^3 = (t+1)^4(t-\mathrm{i})^3(t+\mathrm{i})^3$. According to Theorem~\ref{th:decomposition-curve}, it is contained in the direct sum of spaces
  \begin{equation*}
    \mathcal{X}_{{\mathcal A},-1}^{-4,4} \oplus
    \mathcal{X}_{{\mathcal A},\mathrm{i}}^{-3,4} \oplus
    \mathcal{X}_{{\mathcal A},-\mathrm{i}}^{-3,4} \oplus
    \mathcal{P}_{{\mathcal A}}^6.
  \end{equation*}
  The space $\mathcal{X}_{{\mathcal A},-1}^{-4,4} = {\mathcal Q}_{{\mathcal A},-1}^{-4,3} \oplus
  {\mathcal Q}_{{\mathcal A},-1}^{-3,4}$ is of dimension two. Two basis vectors, normalized
  according to Equation~\eqref{eq:6} and Lemma~\ref{lem:qbasis}, are
  \begin{multline*}
    {\mathbf q}_{{\mathcal A},-1}^{-4} = \begin{psmallmatrix}650.0000\\1680.0000\\4200.0000\end{psmallmatrix}(t+1)^{-4}+\begin{psmallmatrix}227.7112\\-7935.7618\\-12292.7380\end{psmallmatrix}(t+1)^{-3}+\begin{psmallmatrix}-2543.7505\\13075.4877\\12336.2084\end{psmallmatrix}(t+1)^{-2}\\+\begin{psmallmatrix}2459.4830\\-9040.4230\\-4284.7317\end{psmallmatrix}(t+1)^{-1}+\begin{psmallmatrix}-15.0255\\151.7744\\250.1058\end{psmallmatrix}(t+1)+\begin{psmallmatrix}17.8045\\-23.2667\\48.4685\end{psmallmatrix}(t+1)^{2}+\begin{psmallmatrix}7.6851\\-27.3247\\-12.5238\end{psmallmatrix}(t+1)^{3}     \in {\mathcal Q}_{{\mathcal A},-1}^{-4,3}
  \end{multline*}
  and
  \begin{multline*}
    {\mathbf q}_{{\mathcal A},-1}^{-3} = \begin{psmallmatrix}650.0000\\1680.0000\\4200.0000\end{psmallmatrix}(t+1)^{-3}+\begin{psmallmatrix}1028.8024\\-6930.7877\\-8836.9694\end{psmallmatrix}(t+1)^{-2}+\begin{psmallmatrix}-2208.9669\\8567.1053\\4344.5983\end{psmallmatrix}(t+1)^{-1}\\+\begin{psmallmatrix}1172.7946\\2527.6805\\6710.7252\end{psmallmatrix}(t+1)+\begin{psmallmatrix}305.1473\\-4629.6686\\-6862.3617\end{psmallmatrix}(t+1)^{2}+\begin{psmallmatrix}-612.2241\\2925.2016\\2446.0331\end{psmallmatrix}(t+1)^{3}+\begin{psmallmatrix}161.6075\\-574.6046\\-263.3604\end{psmallmatrix}(t+1)^{4}     \in {\mathcal Q}_{{\mathcal A},-1}^{-3,4}.
  \end{multline*}
  The spaces $\mathcal{X}_{{\mathcal A},\pm\mathrm{i}}^{-3,4} =
  {\mathcal Q}_{{\mathcal A},\pm\mathrm{i}}^{-3,4}$ are of dimension one. Basis vectors are
  \begin{multline*}
    {\mathbf q}_{{\mathcal A},\pm \mathrm{i}}^{-3} = \begin{psmallmatrix}-590.0000\mp480.0000\mathrm{i}\\1080.0000\pm1320.0000\mathrm{i}\\-1120.0000\pm1560.0000\mathrm{i}\end{psmallmatrix}(t\mp\mathrm{i})^{-3}+\begin{psmallmatrix}-350.3232\pm1055.6869\mathrm{i}\\3323.1420\mp2332.9210\mathrm{i}\\3307.6750\pm1083.3722\mathrm{i}\end{psmallmatrix}(t\mp\mathrm{i})^{-2}\\+\begin{psmallmatrix}588.4201\pm1072.8139\mathrm{i}\\-1838.5153\mp2746.0136\mathrm{i}\\-890.3356\mp804.3343\mathrm{i}\end{psmallmatrix}(t\mp\mathrm{i})^{-1}+\begin{psmallmatrix}-2040.5475\pm2895.1889\mathrm{i}\\5562.8972\mp5788.0593\mathrm{i}\\7858.8067\pm4083.6954\mathrm{i}\end{psmallmatrix}(t\mp\mathrm{i})+\begin{psmallmatrix}3288.5986\pm526.4819\mathrm{i}\\-7475.2970\mp6041.6541\mathrm{i}\\1723.5378\mp9286.0678\mathrm{i}\end{psmallmatrix}(t\mp\mathrm{i})^{2}\\+\begin{psmallmatrix}347.1064\mp1476.0170\mathrm{i}\\-2866.5767\pm4705.3719\mathrm{i}\\-3724.7627\pm1355.1343\mathrm{i}\end{psmallmatrix}(t\mp\mathrm{i})^{3}+\begin{psmallmatrix}-352.4984\mp117.1860\mathrm{i}\\1253.3277\pm416.6614\mathrm{i}\\574.4419\pm190.9698\mathrm{i}\end{psmallmatrix}(t\mp\mathrm{i})^{4} .
  \end{multline*}
  Since we are interested in real solutions, we replace them by the \emph{real} basis vectors ${\mathbf a}_{\pm\mathrm{i}} \coloneqq \frac{1}{2}({\mathbf q}_{{\mathcal A},+\mathrm{i}}^{-3} + {\mathbf q}_{{\mathcal A},- \mathrm{i}}^{-3})$ and $\mathbf{b}_{\pm\mathrm{i}} \coloneqq \frac{\mathrm{i}}{2}({\mathbf q}_{{\mathcal A},+\mathrm{i}}^{-3} - {\mathbf q}_{{\mathcal A},- \mathrm{i}}^{-3})$ of $\mathcal{X}_{{\mathcal A},\mathrm{i}}^{-3,4} + \mathcal{X}_{{\mathcal A},-\mathrm{i}}^{-3,4}$. Finally, we construct a basis of $\mathcal{P}_{{\mathcal A}}^6$ consisting of three constant vectors, for example $\mathbf{x} = (1,0,0)$, $\mathbf{y} = (0,1,0)$, and $\mathbf{z} = (0,0,1)$, plus two more vectors which can be computed in the standard way by integration:
  \begin{multline*}
    {\mathbf p}_5 \coloneqq
    \int {\mathcal A}\qi\Cj{\mathcal A} \,\mathrm{d}t =
    \begin{psmallmatrix}100\\0\\0\end{psmallmatrix}t+\begin{psmallmatrix}-220\\120\\-160\end{psmallmatrix}t^{2}+\begin{psmallmatrix}140\\-40\\520\end{psmallmatrix}t^{3}+\begin{psmallmatrix}10\\-270\\-470\end{psmallmatrix}t^{4}+\begin{psmallmatrix}-54\\192\\88\end{psmallmatrix}t^{5}     \in \mathcal{P}_{{\mathcal A}}^5 \subseteq \mathcal{P}_{{\mathcal A}}^6
  \end{multline*}
  and
  \begin{multline*}
    {\mathbf p}_6 \coloneqq
    2\int t{\mathcal A}\qi\Cj{\mathcal A} \,\mathrm{d}t =
    \begin{psmallmatrix}100\\0\\0\end{psmallmatrix}t^{2}+\begin{psmallmatrix}-880/3\\160\\-640/3\end{psmallmatrix}t^{3}+\begin{psmallmatrix}210\\-60\\780\end{psmallmatrix}t^{4}+\begin{psmallmatrix}16\\-432\\-752\end{psmallmatrix}t^{5}+\begin{psmallmatrix}-90\\320\\440/3\end{psmallmatrix}t^{6}     \in \mathcal{P}_{{\mathcal A}}^6.
  \end{multline*}
  (The factor $2$ here has the purpose to normalize ${\mathbf p}_6$ in the sense of Lemma~\ref{lem:qbasis} and with respect to $\beta = 0$.) Now, Theorem~\ref{th:decomposition-curve} predicts existence of unique coefficients $\sigma_1$, $\sigma_2$, \ldots, $\sigma_9$ such that ${\mathbf r}= \sigma_1 {\mathbf q}_{{\mathcal A},-1}^{-4} + \sigma_2 {\mathbf q}_{{\mathcal A},-1}^{-3} + \sigma_3 {\mathbf a}_{\pm\mathrm{i}} + \sigma_4 \mathbf{b}_{\pm\mathrm{i}} + \sigma_5 \mathbf{x} + \sigma_6 \mathbf{y} + \sigma_7 \mathbf{z} + \sigma_8 {\mathbf p}_5 + \sigma_9 {\mathbf p}_6$. Indeed, we are able to compute
  \begin{gather*}
    \sigma_{1} = -0.023272,\quad \sigma_{2} = -0.010124,\quad \sigma_{3} = -0.000932,\quad \sigma_{4} = -0.000211,\quad \\\sigma_{5} = 0.000031,\quad \sigma_{6} = 0.000542,\quad \sigma_{7} = 0.010013,\quad \sigma_{8} = 1.523802,\quad \sigma_{9} = 4.169569 .
  \end{gather*}
\end{example}

Theorem~\ref{th:decomposition-curve} allows us to obtain a partial fraction decomposition for rational PH curves. The classical partial fraction decomposition is defined for rational functions but it easily lifts to rational parametrizations. One possible formulation over the complex numbers $\C$ is

\begin{theorem}[Partial Fraction Decomposition]
  For a given rational parametric curve ${\mathbf r}= \frac{-2{\mathbf b}}{\alpha}$ with $\alpha = \prod_{i=1}^k (t-\beta_i)^{n_i}$ there exists a unique polynomial $\mathbf{p}$ and unique rational parametric curves
  \begin{equation*}
    {\mathbf s}_i = \frac{{\mathbf p}_i}{(t-\beta_i)^{n_i}}
    \quad\text{with $\mathbf{p}_i \in \R[t]$, $\deg {\mathbf p}_i < n_i$}
    \quad\text{for}\quad
    i \in \{1, 2, \ldots, k\}
  \end{equation*}
  such that ${\mathbf r}= {\mathbf p}+ \sum_{i=1}^k {\mathbf s}_i$.
\end{theorem}

Note that it is the requirement $\deg {\mathbf p}_i < n_i$ that makes the decomposition unique. As an immediate corollary to Theorem~\ref{th:decomposition-curve}, and still rather surprisingly, this degree constraint can be substantially lifted when we require that $\mathbf{r}$ and all partial fractions are PH curves to the same tangent indicatrix. For our formulation we assume, as usual, that $\mathbf{F} = \mathcal{A} \qi \Cj{\mathcal{A}}$ is fixed and we refer to rational or polynomial curves that solve \eqref{eq:1} simply as ``PH solution curves''.

\begin{corollary}[Partial Fraction Decomposition for Rational PH Curves]
  \label{cor:PFD}
  For any rational PH solution curve ${\mathbf r}= \frac{-2{\mathbf b}}{\alpha}$ with $\alpha = \prod_{i=1}^k (t-\beta_i)^{n_i}$, there exist a unique polynomial PH solution curve $\mathbf{p}$ and unique rational PH solution curves
  \begin{equation*}
    {\mathbf s}_i = \frac{{\mathbf p}_i}{(t-\beta_i)^{n_i}},
    \quad\text{with $\mathbf{p}_i \in \R[t]$, $\deg {\mathbf p}_i \le \deg {\mathbf F}$}
    \quad\text{for}\quad
    i \in \{1, 2, \ldots, k\},
  \end{equation*}
  such that ${\mathbf r}= {\mathbf p}+ \sum_{i=1}^k {\mathbf s}_i$.
\end{corollary}

Our final theorem gives a bound for the degree of a rational PH curve in a solution space that is decomposed similarly to Theorem~\ref{th:decomposition-curve}.

\begin{theorem}
  Consider $\mathbf{r} \in \bigoplus_{i=1}^k \mathcal{X}_{{\mathcal A},\beta_i}^{m_i,M_i} \oplus \mathcal{P}_{{\mathcal A}}^N$ with $m_1$, $m_2$, \ldots, $m_k < 0$ and let $\alpha = \prod_{i=1}^k (t-\beta_i)^{n_i}$ with $n_i = -m_i$. Then there exists a vector valued polynomial ${\mathbf b}$ of degree at most $ n_1 + n_2 + \cdots + n_k + \max\{N, 2a\}$ such that ${\mathbf r}= \frac{-2{\mathbf b}}{\alpha}$.
\end{theorem}

\begin{proof}
  Note again that the space $\mathcal{P}_{{\mathcal A}}^N$ consists only of
  constants if $N<2a+1$. Obviously we have
  \begin{equation*}
    {\mathbf r}=\sum_{i=1}^k\frac{{\mathbf b}_i}{(t-\beta_i)^{n_i}}+ {\mathbf p}
  \end{equation*}
  where ${\mathbf b}_i$ are polynomials of degree at most $n_i+2a$ and $\mathbf{p}$ is of degree at most $N$ (and in fact equal to $0$ if $N<2a+1$). The upper bound on the degree of ${\mathbf b}$ follows directly from writing this expression with the common denominator $\alpha$.
\end{proof}

\section{Proofs of Lemmas}
\label{sec:proofs}

The key element in our proofs is a careful analysis of the system of linear equations arising from \eqref{eq:4}. Let us fix the single root $\beta$ of the denominator polynomial $\alpha$ for the time being and let us denote the solution space of \eqref{eq:4} for the special case $\alpha = (t - \beta)^n$ and $\deg {\mathbf b} \le N$ by $\mathcal{B}_{{\mathcal A},\beta}^{n,N}$. For any ${\mathbf b} \in \mathcal{B}_{{\mathcal A},\beta}^{n,N}$, the curve ${\mathbf r}= \frac{-2{\mathbf b}}{(t-\beta)^n}$ is a rational PH solution curve.

Writing ${\mathbf b}$, $\mathbf{F} = \mathcal{A} \qi \Cj{\mathcal{A}}$, and $\mu$ in the basis $(1,t-\beta, (t-\beta)^2, \ldots)$ as ${\mathbf b} = \sum_{i=0}^\infty {\mathbf b}_i (t-\beta)^i$, $\mathbf{F} = \sum_{i=0}^\infty {\mathbf f}_i (t-\beta)^i$, and $\mu=\sum_{i=0}^\infty \mu_i(t-\beta)^i$ with ${\mathbf b}_i = 0$ for $i > N$, ${\mathbf f}_i = 0$ for $i > 2a$, and $\mu_i \neq 0$ for only finitely many indices, the system \eqref{eq:4} can be written in equivalent form as
\begin{equation}
  \label{eq:linear-system}
  (k-2n){\mathbf b}_{k-n} = \sum_{i=0}^{k-1} \mu_i {\mathbf f}_{k-i-1},\quad
  k \ge 1
\end{equation}
where ${\mathbf b}_{k-n} = 0$ for ${k-n} < 0$. This is a specialization of \cite[Equation~(9)]{kalkan22} to the denominator $(t-\beta)^n$. For $k$ fixed, we refer to \eqref{eq:linear-system} as the ``$k$-th equation''.

The system \eqref{eq:linear-system} is very structured. The first $n-1$ equations are of the triangular shape
\begin{equation}
  \label{eq:linear-system-A}
  \begin{aligned}
    0 &= \mu_0 {\mathbf f}_0, \\
    0 &= \mu_0 {\mathbf f}_1 + \mu_1 {\mathbf f}_0, \\
    \vdots & \\
    0 &= \mu_0{\mathbf f}_{n-2} + \cdots + \mu_{n-3} {\mathbf f}_1 + \mu_{n-2} {\mathbf f}_0.
  \end{aligned}
\end{equation}
The equations to $n \le k \le N+n$ read as
\begin{equation}
  \label{eq:linear-system-B}
  \begin{aligned}
    (-n) {\mathbf b}_0 &= \mu_0{\mathbf f}_{n-1} + \cdots + \mu_{n-2} {\mathbf f}_1 + \mu_{n-1} {\mathbf f}_0, \\
    (-n+1) {\mathbf b}_1 &= \mu_0{\mathbf f}_n + \cdots + \mu_{n-2} {\mathbf f}_2 + \mu_{n-1} {\mathbf f}_1 + \mu_n {\mathbf f}_0, \\
    \vdots & \\
    (-n+N) {\mathbf b}_N &= \mu_0{\mathbf f}_{n+N-1} + \cdots + \mu_{n+N-3} {\mathbf f}_2 + \mu_{n+N-2} {\mathbf f}_1 + \mu_{n+N-1} {\mathbf f}_0.
  \end{aligned}
\end{equation}
Finally, the remaining equations have again zeros on the left-hand side:
\begin{equation}
  \label{eq:linear-system-C}
  \begin{aligned}
    0 &= \mu_0{\mathbf f}_{n+N} + \cdots + \mu_{n+N-1} {\mathbf f}_1 + \mu_{n+N} {\mathbf f}_0, \\
    0 &= \mu_0{\mathbf f}_{n+N+1} + \cdots + \mu_{n+N-1} {\mathbf f}_2 + \mu_{n+N} {\mathbf f}_1 + \mu_{n+N+1} {\mathbf f}_0, \\
    \vdots &
  \end{aligned}
\end{equation}

By \eqref{eq:linear-system}, the left hand side of the $(2n)$-th equation is zero. We will call it the \emph{critical equation.} It appears in \eqref{eq:linear-system-B} if $N \le n$ and in \eqref{eq:linear-system-C} otherwise.

If $n = 1$, the block of equations \eqref{eq:linear-system-A} is empty. Otherwise these equations because of ${\mathbf f}_0={\mathbf F}(\beta)\neq 0$ are equivalent to
\begin{equation}
  \label{eq:17}
  \mu_0=\mu_1=\ldots=\mu_{n-2}=0.
\end{equation}

Let us now consider the equations \eqref{eq:linear-system-C}. Because both $\mu_i$ and ${\mathbf f}_i$ are zero for $i$ high enough, the $\ell$-th equation of \eqref{eq:linear-system-C} is satisfied automatically for $\ell$ high enough. Assume now that $k$ is the maximal number of an equation in \eqref{eq:linear-system-C} with non-zero right-hand side. It has only one non-zero term $\mu_{k-1-2a}{\mathbf f}_{2a}$ forcing $\mu_{k-1-2a}=0$. Decreasing the equation number further, we see that \eqref{eq:linear-system-C} is equivalent to
\begin{equation}
  \label{zeroMU}
  \mu_{n+N-2a}=\mu_{n+N-2a+1}=\mu_{n+N-2a+2}= \ldots =0,
\end{equation}
meaning in particular that all $\mu_i$ vanish if~$n+N\le 2a$.

Consequently the complete system of equations is reduced to \eqref{eq:linear-system-B} with possibly non-zero variables
\begin{equation}
  \label{freeMU}
  \mu_{n-1}, \mu_n, \ldots, \mu_{n+N-2a-1}.
\end{equation}
Clearly, a non-trivial solution requires $N\ge 2a$. The equations in \eqref{eq:linear-system-B} express ${\mathbf b}_i$ explicitly in terms of $\mu_j$ and the coefficients of $\mathbf{F}$. The only exception is possibly the critical equation which has the form
\begin{equation}
  \label{eq:critical}
  0\cdot {\mathbf b}_n=\mu_0{\mathbf f}_{2n-1} + \cdots + \mu_{2n-1} {\mathbf f}_0.
\end{equation}
From this we see that the system can now be solved in the following way: The variable ${\mathbf b}_n$ can be set freely because it appears with a zero coefficient. The remaining variables ${\mathbf b}_i$ are completely determined by the variables $\mu_j$. Variables $\mu_j$ appearing in \eqref{eq:17} and in \eqref{zeroMU} must be zero. The only condition on the remaining variables $\mu_j$ listed in \eqref{freeMU} is, possibly, the critical Equation~\eqref{eq:critical}.

\begin{proof}[Proof of Lemmas~\ref{lem:M0-existence}--\ref{lem:lowCoef}]
  The lemmas talk about the spaces ${\mathcal Q}_{{\mathcal A},\beta}^{m,M}$. Their statement are easily seen to be true in the case $m \ge 0$ because then all elements of ${\mathcal Q}_{{\mathcal A},\beta}^{m,M}$ are polynomials. Therefore, we assume $m < 0$.

  We wish to relate the spaces ${\mathcal Q}_{{\mathcal A},\beta}^{m,M}$ and $\mathcal{B}_{{\mathcal A},\beta}^{n,N}$. A PH curve $\mathbf{r}$ is an element of ${\mathcal Q}_{{\mathcal A},\beta}^{m,M}$ if and only if there exist ${\mathbf b} \in \mathcal{B}_{{\mathcal A},\beta}^{n,N}$ such that ${\mathbf r}= \frac{-2{\mathbf b}}{(1-\beta)^n}$ where $n = -m$, $N = M - m$, and ${\mathbf b}_n = 0$. Consequently, the spaces $\mathcal{B}_{{\mathcal A},\beta}^{n,N}$ and ${\mathcal Q}_{{\mathcal A},\beta}^{m,M}$ are of equal dimension and, in order to prove Lemma~\ref{lem:M0-existence}, we have to show existence of a value $N_0$ such that $\dim\mathcal{B}_{{\mathcal A},\beta}^{n,N} = 0$ for $N < N_0$ and $\dim\mathcal{B}_{{\mathcal A},\beta}^{n,N_0} = 1$. Let us distinguish two cases.

\begin{itemize}
\item If $n > 2a$ we set $N_0 = 2a$. For $N < N_0$, all solutions are trivial by \eqref{freeMU}. For $N = N_0$, however there is precisely one free variable $\mu_{n-1}$. It can be set to any value, providing a one-dimensional system of solutions because the critical equation does not appear in \eqref{eq:linear-system-B}.
\item If $n\leq 2a$ the first non-trivial solution can occur only for $N \ge 2a \ge n$. This places the critical $2n$-th equation into \eqref{eq:linear-system-B}. Let us study the precise form of this equation (including all free $\mu_i$) for increasing $N$:
  \begin{equation*}
    \begin{aligned}
      N &= 2a\colon   & 0 &= \mu_{n-1}{\mathbf f}_n,\\
      N &= 2a+1\colon & 0 & =\mu_{n-1}{\mathbf f}_n+\mu_{n}{\mathbf f}_{n-1},\\
      N &= 2a+2\colon & 0 & =\mu_{n-1}{\mathbf f}_n+\mu_{n}{\mathbf f}_{n-1}+\mu_{n+1}{\mathbf f}_{n-2},\\
      N &= 2a+3\colon & 0 & =\mu_{n-1}{\mathbf f}_n+\mu_{n}{\mathbf f}_{n-1}+\mu_{n+1}{\mathbf f}_{n-2}+\mu_{n+2}{\mathbf f}_{n-3}.
    \end{aligned}
  \end{equation*}
Here it is important that ${\mathbf f}_i$ are vectors of dimension three. If ${\mathbf f}_n=0$ (a non-generic case), we get $N_0=2a$. Otherwise, if ${\mathbf f}_n, {\mathbf f}_{n-1}$ are linearly dependent (non-generic as well), we get $N_0=2a+1$. Otherwise, if ${\mathbf f}_n, {\mathbf f}_{n-1}, {\mathbf f}_{n-2}$ are linearly dependent (still non-generic), we get $N_0=2a+2$. In the (generic) remaining case we get $N_0=2a+3$. This proves Lemma~\ref{lem:M0-existence}.
\end{itemize}

Above discussion also provides the value of $N_0$. The case $n = 1$ requires additional consideration: As the term $\mu_{n+1}{\mathbf f}_{n-2}$ in the case $N = 2a+2$ does not appear we get $N_0(1)=2a+2$. The second case provides $N_0(n)=2a+3$ for $2 \le n \le 2a$. The first case gives $N_0(n)=2a$ if $2a+1 \le n$. These cases directly translate to the first three lines of \eqref{eq:7}. The remaining two lines belong to polynomial cases and can be seen easily. This concludes the proof of Lemma~\ref{lem:M0}.

In order to study the lowest coefficient of an element of ${\mathcal Q}_{{\mathcal A},\beta}^{m,M}$ let us stress that due to \eqref{eq:17} the $n$-th equation becomes $(-n){\mathbf b}_0=\mu_{n-1}{\mathbf f}_0$. Therefore ${\mathbf r}_m={\mathbf b}_0$ is a multiple of ${\mathbf f}_0={\mathbf F}(\beta)$ and Lemma~\ref{lem:lowCoef} is proven.

If $n = 1$ the critical second equation implies $\mu_0 = \mu_1 = 0$ and the first equation provides ${\mathbf r}_{-1}={\mathbf b}_0 = 0$. Hence we have no rational solution. If $n = 2$ the third critical equation implies $\mu_1 = \mu_2 = \mu_3 = 0$ and we also have $\mu_0 = 0$ by \eqref{eq:17}. This implies ${\mathbf r}_{-2}={\mathbf b}_0 = 0$ and ${\mathbf r}_{-1}={\mathbf b}_1 = 0$. No normalized Laurent PH cut can therefore begin with non-zero coefficient at powers $-2$, $-1$ nor $0$ (due to normalization). On the other hand for $n \ge 3$ the variable $\mu_{n-1}$ can be always set to any non-zero value. In particular the choice $\mu_{n-1}=-n$ leads to ${\mathbf r}_m={\mathbf b}_0={\mathbf f}_0={\mathbf F}(\beta)$ which concludes the proof of Lemma~\ref{lem:qbasis}.
\end{proof}

Let us now turn our attention to the unique open lemma on curves with several roots in the denominator. Somewhat surprisingly, the system of equations for the single root case turns out to be again very useful.

\begin{proof}[Proof of Lemma~\ref{lem:finL}]
  The statement on the Laurent expansion \eqref{eq:14} is clear. We set $\tilde \alpha(t) \coloneqq (t-\beta)^n$. There exists a polynomial $\hat{\alpha}(t)$ such that $\alpha(t) = \tilde \alpha(t) \hat{\alpha}(t)$ and a vector valued polynomial ${\mathbf b}(t)$ such that
  \begin{equation*}
    {\mathbf r}(t) = \frac{-2{\mathbf b}(t)}{\alpha(t)}.
  \end{equation*}
  Moreover, there exists a polynomial $\mu(t)$ so that \eqref{eq:4} is satisfied. We are going to construct a solution $\tilde \mu(t)$, $\tilde {\mathbf b}(t)$ to its specialization \eqref{eq:linear-system} that fulfills the requirements of the lemma.

  There exists a vector-valued polynomial $\tilde {\mathbf b}(t)$ satisfying
  \begin{equation}
    \label{eq:taylor}
    \tilde {\mathbf b}(t)=\frac{{\mathbf b}(t)}{\hat{\alpha}(t)}+\mathcal O(t-\beta)^{(n+1)}.
  \end{equation}
  The above equation uniquely determines the coefficients $\tilde {\mathbf b}_0, \tilde {\mathbf b}_1, \ldots, \tilde {\mathbf b}_n$ of $\tilde {\mathbf b}$ while the remaining coefficients are yet to be chosen. Nonetheless, a straightforward computation already provides
  \begin{equation*}
    \tilde \alpha'\tilde{{\mathbf b}}-\tilde \alpha \tilde{{\mathbf b}}'=\frac{(\tilde \alpha \hat{\alpha})'{\mathbf b}-(\tilde \alpha \hat{\alpha}){\mathbf b}'}{\hat{\alpha}^2}+\mathcal O(t-\beta)^{2n}=\frac{\mu \mathbf{F}}{\hat{\alpha}^2}+\mathcal O(t-\beta)^{2n}=\tilde \mu \mathbf{F}+\mathcal O(t-\beta)^{2n}
  \end{equation*}
  where $\tilde \mu$ is the $(2n-1)$-th Taylor polynomial of $\frac{\mu}{\hat{\alpha}^2}$ at $t = \beta$. Consequently the first $2n$ equations of the system \eqref{eq:linear-system} are satisfied by $\tilde {\mathbf b}$ and $\tilde \mu$. Because these include the critical $2n$-th equation, the solution for the yet undetermined coefficients of $\tilde {\mathbf b}$ can be obtained by setting $\tilde\mu_i=0$ for all indices $i\geq2n$.

  In order to bound the degree of the thus obtained solution, we have to count the number of equations with non-vanishing right-hand side. Our choice for $\tilde\mu_i$, $i \geq 2n$ implies that all equations starting from $k = 2(n+a)$ have vanishing right-hand side so that all coefficients $\tilde {\mathbf b}_i$ with $i > n+2a$ are zero whence $\deg\tilde {\mathbf b} \le n + 2a$.

  With this solution, the Laurent series of $\tilde {\mathbf r}(t) \coloneqq \frac{-2\tilde {\mathbf b}(t)}{\tilde \alpha(t)}$ at $t = \beta$ indeed starts with index $m=-n$ and ends with $M=2a$. Due to \eqref{eq:taylor}, the coefficients of non-positive powers in the Laurent series of ${\mathbf r}$ and $\tilde {\mathbf r}$ at $t = \beta$ coincide.
\end{proof}

\section{Conclusion and Future Work}
\label{sec:conclusion}

Our results in this article give a fairly complete insight into the structure of the solution spaces to Problem~\ref{pr1}. We have explicitly and in detail described the PH case and the generic roots of the denominator. Together with the proofs and remarks however a more general description transpires. Considering the vector space $\mathcal{R}$ of rational solution curves to a general (not necessarily PH) tangent indicatrix $\mathbf F(t)$, we obtain
\begin{equation}
  \label{eq:18}
  \mathcal {R}=\bigoplus_{\beta \in \mathbb C}\left (\bigoplus_{i=-\infty}^{-3} {\mathcal Q}_{\beta}^{i,M_0(i)}\right )
+
  \bigoplus_{\beta \in \mathcal S} {\mathcal Q}_{\beta}^{-2,M_0(-2)}
+
  \bigoplus_{\beta \in \mathcal S_2} {\mathcal Q}_{\beta}^{-1,M_0(-1)}\oplus  \mathcal P
 \end{equation}
where $\mathcal S$ is the set of the roots of $\det[\,{\mathbf F}(t),\mathbf F'(t),\mathbf F''(t)\,]$ and $\mathcal S_2\subset \mathcal S$ is the set of its double-roots. While the first of these sets will be generically finite, the second one will be typically empty.

Specializations of \eqref{eq:18} for bounded degree of the numerator and/or fixed denominator polynomial $\alpha$ are possible and result in canonical basis constructions for several vector spaces of interest. These bases are easy to compute and provide a lot of insight into the respective solution spaces. One aspect is the computationally efficient generalized partial fraction decomposition of Corollary~\ref{cor:PFD}.

Future research will focus on using our basis construction to solve
interpolation problems with rational PH curves \cite{jaklic11,jaklic14}. We
believe that a crucial advantage over existing envelope approaches is the
possibility to avoid unwanted shape properties (cusps) of the interpolant. The
seemingly superfluous polynomial $\mu$ of \eqref{eq:4} plays a crucial role in
that context. There are also some preliminary results that PH curves with the
additional property of having a rational arc length function
\cite{FaroukiSakkalis2019} can be nicely subsumed to our approach. We also hope
that our method can be useful for the construction of PN surfaces. We are able
to compute some examples but it is more difficult to understand the more
difficult structure of the problem.
 
\bibliographystyle{plainnat}

\end{document}